\documentclass{amsart}
\usepackage{graphicx}
\usepackage{graphicx}
\usepackage{color}
\usepackage{times}
\usepackage{cite}
\usepackage{url}
\usepackage{enumerate,latexsym}
\usepackage{latexsym}
\usepackage{amsmath,amssymb}
\usepackage{graphicx}
\usepackage{amsthm}
\usepackage{verbatim}
\newtheorem{thm}{Theorem}[section]
\newtheorem*{theorem*}{Theorem}
\newtheorem*{acknowledgement*}{Acknowledgement}
\newtheorem{cor}[thm]{Corollary}

\newtheorem{lem}[thm]{Lemma}
\newtheorem{prop}[thm]{Proposition}
\theoremstyle{definition}

\theoremstyle{remark}
\newtheorem{rem}[thm]{Remark}

\numberwithin{equation}{section}

\newcommand{\set}[1]{\left\{#1\right\}}
\newcommand{\Real}{\mathbb R}

\title{Orientable minimal submanifolds of low index in the sphere} 
 \author{Jacob Bernstein}\address{Johns Hopkins University \\ 3400 N. Charles St, Baltimore MD 21218 USA}
	\email{jberns15@jhu.edu}
\author{Daniel Ketover}\address{Rutgers University\\  Busch Campus - Hill Center \\ 110 Freylinghausen Road, Piscataway NJ 08854 USA}

 \email{dk927@math.rutgers.edu}

\begin{document}
\maketitle
\begin{abstract}
In this short note we prove a rigidity result for orientable minimal submanifolds in the round sphere of the lowest non-trivial Morse index.  In particular,  we show that the only closed and connected $(2m-2)$-dimensional orientable minimal submanifolds in $\mathbb{S}^{2m}$ of index less than or equal to $2m+1$ are the totally geodesic spheres.
\end{abstract}

\section{Introduction}
We study minimal submanifolds in the round sphere with low Morse index.   Our focus is on closed oriented submanifolds of codimension two.  Simons \cite{simonsMinimalVarietiesRiemannian1968} showed the totally geodesic spheres $\mathbb{S}^k\subset \mathbb{S}^{n}\subset \Real^{n+1}$ possess the lowest possible Morse indices, namely $n-k$, and they are characterized by this property.  This follows from the fact that the conformal vector fields, i.e., the tangential  projection to $\mathbb{S}^n$ of parallel vector fields on $\Real^{n+1}$, are area non-increasing deformations of any closed minimal submanifold.
A further consequence is that any closed smooth non-spherical minimal submanifold in $\mathbb{S}^n$ has index at least $n+1$. 

In general this bound is sharp, the Veronese, $\Sigma_{\mathbb{RP}}$, is the non-orientable minimal surface given by an explicit codimension two embedding 
$$
\mathbb{RP}^2\to\Sigma_\mathbb{RP}\subset  \mathbb{S}^4.
$$
It has Morse index five and this characterizes it among minimal projective planes \cite{karpukhinIndexMinimalSpheres2021, kusnerIndexMinimal2tori2024}.
In contrast, we show that in even dimensional spheres this bound is never achieved by any orientable minimal submanifold of codimension two.   In odd dimensions our methods give a weaker conclusion -- namely,  if an orientable minimal submanifold of codimension two achieves the bound, then, after rotating, it is the link of a minimal cone that is a complex submanifold.  We do not address whether this can occur in a non-trivial way.

More precisely, our main result is the following:
\begin{thm}\label{MainThm}
	For $n\geq 3$, let $\Sigma^{n-2}\subset \mathbb{S}^n$ be a closed smooth oriented and connected minimal submanifold, possibly immersed.  One of the following holds:
	\begin{enumerate}
		\item $\Sigma$ is a totally geodesic sphere of Morse index two;
		\item $n\geq 5$ is odd, the Morse index of $\Sigma$ is $n+1$ and, away from the origin, $C(\Sigma)$, the cone over $\Sigma$,  is a complex, possibly immersed, submanifold         for some parallel complex structure on $\Real^{n+1}$;
		\item $\Sigma$ has Morse index at least $n+2$.
	\end{enumerate}
\end{thm}
\begin{rem}
We can say a little bit more, namely, that either  the stability operator,  $-\mathcal{L}_\Sigma$, admits an eigensection with eigenvalue $<-(n-2)$  and the index is at least $n+2$ or:
	\begin{enumerate}
		\item $\Sigma$ is a totally geodesic sphere;
		\item  The curvature term in the stability operator, $\tilde{A}$, has the form $\tilde{A}=\frac{1}{2}|A|^2 \tilde{I}$.
	\end{enumerate}
	When $n=4$, Case (2) occurs only when $\Sigma$ is the double cover of the Veronese.  This cover is an immersed minimal two-sphere of index ten considered by Ejiri \cite{ejiriIndexMinimalImmersions1983}.
\end{rem}
In $\mathbb{S}^4$ this bound is sharp:
\begin{cor} If $\Sigma$ is a closed, connected, oriented minimal surface in $\mathbb{S}^4$ that is not a totally geodesic $\mathbb{S}^2$, then the index is at least six and this bound is achieved by a Clifford torus inside a totally geodesic $\mathbb{S}^3$.
\end{cor}

For hypersurfaces, an analogous result was established by Perdomo \cite{perdomoLowIndexMinimal2001} in all dimensions. For minimal two-spheres in $\mathbb{S}^4$ the result follows from a more general classification of Ejiri \cite{ejiriIndexMinimalImmersions1983} and, more generally,  for superminimal surfaces in $\mathbb{S}^4$ it follows from work of Montiel-Urbano \cite{MontielUrbano}.  For minimal tori in $\mathbb{S}^4$ it is a special case of a  result of Kusner-Wang \cite{kusnerIndexMinimal2tori2024}.  For oriented minimal surfaces in $\mathbb{S}^3$, Urbano \cite{Urbano1990} obtained stronger rigidity: the only minimal surfaces of index less than or equal to five were the Clifford torus, with index five, and the equatorial spheres of index one.  Similarly, Kusner-Wang \cite{kusnerIndexMinimal2tori2024} showed that a minimal torus in $\mathbb{S}^4$ with index at most six is a Clifford torus inside a totally geodesic $\mathbb{S}^3$.

The proof of Theorem \ref{MainThm} uses the existence of a parallel almost complex structure, i.e., a well defined rotation by $90^\circ$, on the normal bundle to $\Sigma$ in an elementary way.  Specifically, it allows one to try and find an additional unstable direction beyond those given by the conformal vector fields by applying this rotation to the conformal vector fields and averaging -- see Lawson-Simons \cite{lawsonStableCurrentsTheir1973} for a related strategy.

When $C(\Sigma)$ is a complex variety it is area minimizing (and hence stable).  In Proposition \ref{LargeIndexVarietyProp}, we construct an example of an immersed $\Sigma$ so $C(\Sigma)$ is an immersed complex submanifold away from the origin, but the Morse index of $\Sigma$ is strictly greater than $n+1$.  It would be interesting to understand how restrictive having low index is in odd dimensional spheres.
Finally, it would be interesting to understand what occurs when $\Sigma$ has (mild) singular set.
  For instance, one expects the spherical suspensions of $\Sigma_\mathbb{RP}$ inside of $\mathbb{S}^{n}$ (which have codimension three singular set) to have Morse index five -- cf. \cite{lawsonStableCurrentsTheir1973}.

\emph{Acknowledgements}:  J.B was partially supported by the NSF grant DMS-2203132.  D.K. was partially supported by NSF grant DMS-2405114.  J.B. would like to thank Zhihan Wang for his interest in this work and for making him aware that an LLM had independently arrived at a proof of some form of the main result.

\section{Background}
We collect some background material.
\subsection{Orientability of $N\Sigma$}

Let $\Sigma\subset \mathbb{S}^{n}$ be a $(n-2)$-dimensional oriented and connected submanifold of $\mathbb{S}^{n}$.  For any $p\in \Sigma$, we can define an endomorphism 
$$J_N:N_p\Sigma\to N_p\Sigma$$
as follows.  For $\nu \in N_p \Sigma$ of unit length let 
$J_N(\nu)$
denote the unique section of the normal bundle so that for any oriented orthonormal basis basis $E_1, \ldots, E_{n-2}$ of $T_p\Sigma$ the collection
$$
\{E_1, \ldots, E_{n-2}, \nu, J_N(\nu)\}
$$
forms an oriented orthonormal basis of $T_p \mathbb{S}^{n}$.
One readily checks that $J_N$ is parallel with respect to the natural connection induced from the sphere.

\subsection{Stability operator of a minimal submanifold}

The stability operator on $\Sigma\subset \mathbb{S}^n$, a $k$-dimensional minimal submanifold, 
is the differential operator acting on sections of the normal bundle, $N\Sigma$, given by
$$
\mathcal{L}_\Sigma=\Delta^\perp +k \tilde{I}+\tilde{A}.
$$
Here $\Delta^\perp$ is the normal Laplacian, $\tilde{I}$ is the identity endomorphism and $\tilde{A}$ is a curvature term. 
For $W$ is a local section of $N\Sigma$ and $E_1, \ldots, E_{k}$ a local orthonormal frame of $T\Sigma$
$$
\Delta^\perp W=\sum_{i=1}^{k}\left( \nabla^\perp_{E_i} \nabla^\perp_{E_i} W- \nabla^\perp_{\nabla_{E_i} E_i} W\right).
$$
The curvature term $\tilde{A}$  satisfies, for sections $V$ and $W$ of the normal bundle,
$$
V\cdot \tilde{A}(W)= \langle A_V, A_W\rangle
= \sum_{i=1}^{k} (V\cdot {A} (E_i,E_i)) (W\cdot {A} (E_i, E_i))
$$
where $A_V$
is the second fundamental form
in the direction $V$.

Associated to $\mathcal{L}_\Sigma$ is the bilinear form satisfying
$$
Q_\Sigma[W, W]=-\int_{\Sigma} W (\mathcal{L}_\Sigma W) = \int_\Sigma |\nabla^\perp W|^2 -k |W|^2 -W\cdot \tilde {A}(W).
$$
The \emph{Morse index}, $\mathrm{Ind}(\Sigma)$, of $\Sigma$ is the maximal dimension of space of sections on which the associated quadratic form is negative definite. It may also be computed by counting the number of negative eigenvalues, counting multiplicity, of $-\mathcal{L}_\Sigma$.

One readily computes that if ${V}$ is a parallel vector field on $\mathbb{R}^{n+1}$ and ${V}^\perp $ is the orthogonal projection of this vector field onto $N\Sigma$ then one has
$$
-\mathcal{L}_\Sigma {V}^\perp = -k {V}^\perp.
$$
That is, when $V^\perp \neq 0$ it is an eigensection of $-\mathcal{L}_\Sigma$ with eigenvalue $-k<0$.

\section{Proof of Theorem \ref{MainThm}}
We prove Theorem \ref{MainThm} by averaging $Q_\Sigma$ over $J_N (V^\perp)$ where $V^\perp$ is the normal component of a parallel vector field $V$ of $\Real^{n+1}$.  For related techniques, see \cite{lawsonStableCurrentsTheir1973,Aminov, MicallefWolfson}.

We first note a useful identity. Fix a point $p\in \Sigma$ and choose $N_1, N_2$, to be an orthonormal frame of $N\Sigma$ near $p$ with 
$$
J_N(N_1)=N_2, J_N(N_2)=-N_1.
$$
If $\tilde{A}$ is the curvature term in the stability inequality then one readily computes that
$$
\tilde{A}(N_1)= |A_{N_1}|^2 N_1+\langle A_{N_1}, A_{N_2}\rangle  N_2, \tilde{A}(N_2)=\langle A_{N_1}, A_{N_2}\rangle N_1+ |A_{N_2}|^2  N_2.
$$
An immediate consequence is that
\begin{equation} \label{Intertwine}
\tilde{A}\circ J_N=J_N\circ \left(-\tilde{A}+|A|^2\tilde{I}\right)
\end{equation}
where here $\tilde{I}$ is the identity endomorphisms on the normal bundle.

\begin{proof}[Proof of Theorem \ref{MainThm}]
When $\Sigma$ is totally geodesic we are in Case (1).  Otherwise, as $\Sigma$ is smooth, the cone of $\Sigma$ can't split a line and so when $V$ is a non-zero parallel vector field on $\Real^{n+1}$,  ${V}^\perp$ is non-trivial.  In particular, the space of eigensections of $-\mathcal{L}_\Sigma$ with eigenvalue $-(n-2)$ has dimension at least $n+1$.  Hence, in this case if there is a section, $W$, so that 
$$Q_{\Sigma}[W,W]<-(n-2)\int_\Sigma |W|^2,
$$
then the Morse index is at least $n+2$ and we are in Case (3). 
As such, in what follows we may suppose $V^\perp$ is non-trivial and $-(n-2)$ is the lowest eigenvalue of $-\mathcal{L}_\Sigma$. 
     
As $V^\perp$ is an eigensection with eigenvalue $-(n-2)$, 
$$
-(n-2)\int_\Sigma |V^\perp|^2 =\int_{\Sigma}	|\nabla^\perp V^\perp |^2-(n-2) |V^\perp|^2-V^\perp \cdot \tilde{A}(V^\perp).
$$
and so
$$
\int_{\Sigma} V^\perp \cdot \tilde{A}(V^\perp)=\int_{\Sigma}	|\nabla^\perp V^\perp |^2.
$$

By the eigenvalue assumption,  for any section $W$ of $N\Sigma$,
$$
-(n-2) \int_\Sigma |W|^2 \leq \int_{\Sigma}	|\nabla^\perp W |^2-(n-2) |W|^2-W \cdot \tilde{A}(W) .
$$
That is, for a non-trivial section, $W$,
$$
\int_\Sigma W \cdot\tilde{A}(W) \leq \int_{\Sigma}	|\nabla^\perp W |^2,
$$
with equality only when $W$ is an eigensection of $-\mathcal{L}_\Sigma$ with eigenvalue $-(n-2)$. 

Fix $E_1, \ldots, E_{n+1}$ an orthonormal frame of parallel vector fields on $\Real^{n+1}$. Using \eqref{Intertwine}, the fact that $J_N$ is parallel and the fact that the $E_i^\perp$ are eigensections of $-\mathcal{L}_\Sigma$, yields
\begin{align*}
	\int_\Sigma |A|^2&=\sum_{i=1}^{n+1}	\int_\Sigma J_N(E_i^\perp)\cdot\tilde{A}(J_N(E_i^\perp))  \leq\sum_{i=1}^{n+1}\int_{\Sigma}	|\nabla^\perp J_N(E_i^\perp)|^2\\
&=\sum_{i=1}^{n+1}\int_{\Sigma}	|\nabla^\perp E_i^\perp|^2=\sum_{i=1}^{n+1}\int_{\Sigma} E_i^\perp \cdot \tilde{A}(E_i^\perp) =	\int_\Sigma |A|^2. 
\end{align*}
That is,  we have equality throughout and so $J_N (V^\perp)$ is a eigensection of $-\mathcal{L}_\Sigma$ with eigenvalue $-(n-2)$ for all parallel non-zero vectors in $\Real^{n+1}$. Using \eqref{Intertwine},
$$
-(n-2) J_N V^\perp= -\mathcal{L}_\Sigma J_N V^\perp=-J_N (\mathcal{L}_\Sigma V^\perp)+|A|^2 J_N V^\perp -2 J_N \tilde{A}(V^\perp)
$$
$$
=-(n-2) J_N V^\perp+|A|^2 J_N V^\perp -2 J_N \tilde{A}(V^\perp).
$$
Hence, for all $V$,
$
\tilde{A}(V^\perp)=\frac{1}{2} |A|^2 V^\perp.
$
and so
$$
\tilde{A}=\frac{1}{2} |A|^2 \tilde{I}.
$$

Finally, by the assumptions from the beginning, either,  for any parallel vector fields, $V$,  $J_N(V^\perp)=W^\perp$ for some parallel vector field, $W$, or the multiplicity of the $-(n-2)$ eigenvalue is at least $n+2$ and we are in Case (3).  If the former occurs, then as $\Sigma$ is not totally geodesic, Lemma \ref{CpxVarietyLem} implies and we are in Case (2).
\end{proof}

\begin{lem}\label{CpxVarietyLem}
	For $n\geq 3$, suppose $\Sigma^{n-2}\subset \mathbb{S}^n$ is smooth closed oriented and connected submanifold, possibly immersed, so that, for every parallel vector field on $\Real^{n+1}$,  $V$,  $J_{N} (V^\perp) =W^\perp$ for some parallel vector field $W$ on $\Real^{n+1}$.  Then, either
	\begin{enumerate}
		\item $\Sigma$ is a totally geodesic sphere;
		\item $n=2m+1$ is odd and away from the origin $C(\Sigma)$ is an immersed complex submanifold for some choice of parallel complex structure.
	\end{enumerate}
\end{lem}
\begin{proof}
	Suppose $V$ is non-zero and $V^\perp$ identically vanishes.  This means that $V$ is tangent to $C(\Sigma)$ everywhere and so $C(\Sigma)$ splits a line.  As $\Sigma$ is smooth this can only happen if $\Sigma$ is a totally geodesic sphere and the claim is trivial in this case.
		Hence, we may assume $V^\perp$ is non-zero for all non-zero $V$. Together with the hypothesis this implies there is a well defined map $L:\Real^{n+1}\to \Real^{n+1}$ determined by $J_N(V^\perp)=(L(V))^\perp$.

We claim that $L$ is linear and $L\circ L=-I$ and so $n$ must be odd.  Indeed,
    \begin{align*}
    (L(V+W)&-L(V)-L(W))^\perp=(L(V+W))^\perp-L(V)^\perp-L(W)^\perp \\
    &=J_N((V+W)^\perp)-J_N(V^\perp)-J_N(W^\perp)=0.
    \end{align*}
By our assumptions this only occurs when $L(V+W)=L(V)+L(W)$.  The other claims follow in the same fashion.

We next show $L$ is orthogonal. To that end pick a point $p_1\in \Sigma$ and choose $N_1^1$ and $N_1^2$ in $\Real^{n+1}$ that form an orthonormal basis of $N_{p_1}\Sigma$ and so $J_N(N_1^1)=N_1^2$.  It follows that $L(N_1^1)=N_1^2$ and $L(N_1^2)=-N_1^1$ and for any $V$ orthogonal to $N_{p_1}\Sigma$, $L(V)$ is also orthogonal to $N_{p_1} \Sigma$.  That is, $L$ respects the orthogonal splitting of $N_{p_1} \Sigma \oplus (N_{p_1}\Sigma)^\perp$ and is an orthogonal transformation on the first factor.

Now pick a $p_2$ so $N_{p_2}\Sigma$ is not the same subspace as $N_{p_1}\Sigma$ -- such a $p_2$ exists as otherwise any $V$ orthogonal to $N_{p_1}\Sigma$ would have $V^\perp=0$ which violates our assumption.  Pick $N_2^1, N_2^2$ that form an orthornormal basis of $N_{p_2} \Sigma$  and so $J_N(N_2^1)=N_2^2$ and hence $L(N_2^1)=N_2^2$.  This implies $N_{p_2}\Sigma\cap N_{p_1} \Sigma=\set{0}$.

Let us write $N_2^i= V^i+W^i$ where $V^i\in N_{p_1}\Sigma$ and $W^i$ is orthogonal to $N_{p_1}\Sigma$.  From the properties of $L$ observed above $L(V^i)\cdot L(V^j)=V^i\cdot V^j$ and $L(V^i)\cdot L(W^j)=0$.  By the choices made and fact that $J_N$, and hence $L$, is orthogonal on $N_{p_2}\Sigma$ we have
$$
(V^i+W^i )\cdot (V^j+W^j)=L(V^i+W^i)\cdot L(V^j+W^j)=V^i \cdot V^j + L(W^i)\cdot L(W^j)
$$
It follows that 
$$
W^i\cdot W^j =L(W^i)\cdot L(W^j).
$$
Hence, $N_{p_1}\Sigma \oplus N_{p_2}\Sigma$ is four dimensional with basis $N_1^1, N_2^1, W^1, W^2$ and $L$ maps it to itself orthogonally.  
We inductively pick points $p_1, \ldots, p_m$ to obtain a decomposition $\Real^{n+1}=N_{p_1}\Sigma \oplus \cdots \oplus N_{p_m} \Sigma$
and inductively verify that $L$ is orthogonal.

	Now fix a point $p\in \Sigma$.  The position vector of $p$, $X(p)\in \mathbb{R}^{n+1}$, satisfies $X(p)^\perp =0$ at $p$.  Hence,  $(L(X(p)))^\perp=0$ at $p$ and so $L(X(p))$ is tangent to $\Sigma$ at $p$.  Similarly,   if $Z\in T_p\Sigma$ is orthogonal to $L(X(p))$, then $L(Z)$ is orthogonal to $X(p)$.  Moreover, as $Z^\perp =0$ we have $(L(Z))^\perp=J_N(Z^\perp)=0$  at $p$ and so $L(Z)\in T_p \Sigma$.    It immediately follows that $L$ preserves the tangent space of $C(\Sigma)$.  That is, $C(\Sigma)$ is a complex submanifold with respect to the parallel complex structure determined by $L$.
\end{proof}

\section{Small eigenvalues of the Laplacian and the Morse index}
In this section we show that a minimal submanifold with many small eigenfunctions of the Laplace operator has high Morse index.  This is used to give an example of an immersed minimal submanifold in $\mathbb{S}^5$ of high Morse index whose cone is an immersed complex submanifold away from the origin.

First, recall that for any $k$-dimensional minimal submanifold $\Sigma^k\subset \mathbb{S}^n\subset \Real^{n+1}$ the coordinate functions $x_i$ restrict to eigenfuctions of $-\Delta_{\Sigma}$ with eigenvalue $k$.  That is,
$$
-\Delta_{\Sigma} x_i = k x_i.
$$
Let $N_\Sigma^*$ be the number of eigenvalues (counting multiplicity) of $-\Delta_\Sigma$ strictly below $k$.  

We show the Morse index is large when there are many small eigenfunctions:
\begin{prop}\label{EigenvalueProp}
	Suppose $\Sigma^{k}\subset \mathbb{S}^n$ is smooth and minimal and $N_{\Sigma}^*> (n+1)^2$, then $\Sigma$ has Morse index strictly greater than $n+1$.
	More generally, if $N_{\Sigma}^*> (n+1)l$ for $l\geq n+1$, then $\Sigma$ has Morse index strictly greater than $l$.
\end{prop}
Combining this with a result of Korevaar \cite{korevaarUpperBoundsEigenvalues1993} yields the following:
\begin{cor}\label{HighIndexCor}
	If $\Sigma^2\subset \mathbb{S}^{n}$ is a closed orientable minimal surface, possible immersed, of area $A$ and genus $g$, then
	$$
	\mathrm{Ind}(\Sigma)\geq \frac{ C A}{(g+1)(n+1)^2}.
	$$
	for some universal constant $C>0$.
	
	In particular, if $\Sigma$ is a minimal torus then any there is a cover, ${\Sigma'}$, of $\Sigma$ of arbitrarily high Morse index.
\end{cor}

Before proving Proposition \ref{EigenvalueProp}
we record a simple computational lemma.
\begin{lem}\label{ComputationLem}
	Let $\Sigma^k \subset \mathbb{S}^n$ be a minimal submanifold, $\phi$ an section of the normal bundle and $f\in C^\infty(\Sigma)$ a smooth function.  We have
	\begin{align*}
	Q_\Sigma[f\phi, f\phi]&=  \int_{\Sigma}f^2 \left( |\nabla^\perp \phi|^2 -k |\phi|^2-\phi \cdot \tilde{A}(\phi) \right) +(-f \Delta_\Sigma f) |\phi|^2\\
	&=\int_{\Sigma} f^2 (-\phi\cdot \mathcal{L}_\Sigma \phi)+ |\nabla f|^2 |\phi|^2.
	\end{align*}
\end{lem}
\

\begin{proof}[Proof of Proposition \ref{EigenvalueProp}]
For any function, $f$, and parallel vector field, $V$,
Lemma \ref{ComputationLem} yields
$$
Q_\Sigma[f V^\perp, f V^\perp]=\int_\Sigma f^2(-V^\perp  \mathcal{L}_\Sigma V^\perp)+|\nabla f|^2 |V^\perp|^2. =\int_{\Sigma} (|\nabla f|^2-k f^2) |V^\perp|^2.
$$

Let $\mathcal{E}(\mu)$ be the space of eigenfuctions of $-\Delta_\Sigma$ with eigenvalue $<\mu$ so $\dim \mathcal{E}(k)=N_{\Sigma}^*$.  Fix an orthonormal basis, $E_1, \ldots, E_{n+1}$,  of $\Real^{n+1}$ and consider the linear map
$$
I: \mathcal{E}(k) \to \Real^{(n+1)\times (n+1)}, f\mapsto \left[\int_{\Sigma} f E_i^\perp \cdot E_j^\perp  \right]_{ij}.
$$
If $\dim \mathcal{E}(k)=N_\Sigma^* >(n+1)^2$, then the kernel of the map is at least one-dimensional and so there is a non-zero element $f\in \mathcal{E}(k)$ with $I(f)=0$.  For some $\mu<k$, $f$ satisfies 
$$
\int_{\Sigma} |\nabla f|^2 -k f^2 \leq (\mu-k) \int_{\Sigma} f^2 <0.
$$
.

If $\mathrm{Ind}(\Sigma)\leq n+1$, then, as $E_i^\perp$ span the space of negative eigensections of $\mathcal{L}_\Sigma$ and, for each $1\leq j \leq n+1$,  $f E_j^\perp$ is orthogonal to every element of this space, 
$$
0\leq Q_\Sigma[f E_j^\perp, f E_j^\perp].
$$
Hence, 
$$
0\leq \sum_{i}\int_\Sigma Q_\Sigma[f E_i^\perp, f E_i^\perp]=(n-k)\int_{\Sigma}|\nabla f|^2 -k  f^2  \leq (n-k) (\mu-k) \int_\Sigma f^2<0.
$$
This is a contradiction and proves that the Morse index is strictly greater than $n+1$.  The general result follows by the same method.
\end{proof}

Finally, we give an example clarifying Case (2) of Theorem \ref{MainThm}.
\begin{prop}\label{LargeIndexVarietyProp}
	There exists a minimally immersed three-manifold, $\Sigma^3\subset \mathbb{S}^5\subset \Real^{6}$, with $\mathrm{Ind}(\Sigma)>6$ and so $C(\Sigma)$ is an immersed complex submanifold away from the origin.
\end{prop}
\begin{proof}
	Fix a genus zero complex curve, $C_0$ in $\mathbb{CP}^2$.  Let $\pi: \mathbb{S}^5\to \mathbb{CP}^2$ be the Hopf fibration.  Note that this means $\pi^{-1}(C_0)=\Sigma$ is a minimal submanifold with $C(\Sigma)$ a complex variety.  By \cite{korevaarUpperBoundsEigenvalues1993}, the Laplace operator of any sufficiently large cover of $C_0$, $C_0'$, has many eigenfunctions of eigenvalue less than $3$.  As the fibers of $\pi$ are geodesics,  a result of Watson \cite{watsonManifoldMapsCommuting1973} implies $\Sigma'=\pi^{-1}(C_0')$ also has many eigenfunctions with eigenvalue less than $3$ and so the Morse index can be made as large as one likes by Proposition \ref{EigenvalueProp}.
\end{proof}
\bibliographystyle{unsrt}
\bibliography{export-data}

\end{document}